\documentclass[sn-mathphys,Numbered]{sn-jnl}
\usepackage{graphicx}%
\usepackage{multirow}%
\usepackage{amsmath,amssymb,amsfonts}%
\usepackage{amsthm}%
\usepackage{mathrsfs}%
\usepackage[title]{appendix}%
\usepackage{xcolor}%
\usepackage{textcomp}%
\usepackage{manyfoot}%
\usepackage{booktabs}%
\usepackage{algorithm}%
\usepackage{algorithmicx}%
\usepackage{algpseudocode}%
\usepackage{listings}%
\theoremstyle{thmstyleone}%
\newtheorem{theorem}{Theorem}%
\newtheorem{proposition}[theorem]{Proposition}%
\theoremstyle{thmstyletwo}%
\usepackage{bm}
\usepackage{float}
\usepackage{booktabs}
\usepackage{epsfig}
\usepackage{graphicx}
\usepackage{caption}
\usepackage{color}
\usepackage{ulem}
\usepackage{relsize}
\usepackage{exscale}
\usepackage{xcolor}
\usepackage{varioref}
\usepackage{comment}
\usepackage{bbm}

\theoremstyle{thmstylethree}%
\usepackage{actuarialangle}

\newtheorem{ex}{Example}

\newtheorem{lem}{Lemma}

\def\annu#1{_{%
\vbox{\hrule height .2pt 
\kern 1pt 
\hbox{$\scriptstyle {#1}\kern 1pt$}%
}\kern-.05pt 
\vrule width .2pt 
}}

\begin{document}

\title[Balducci]{Several facts about Theodor Wittstein, Gaetano Balducci, and some expressions of the net single premiums under their mortality assumption}

\author*[1]{\fnm{Andrius} \sur{Grigutis}}\email{andrius.grigutis@mif.vu.lt}
\author[1]{\fnm{Eglė} \sur{Matulevičiūtė}}\email{egle.matuleviciute@mif.stud.vu.lt}
\author[1]{\fnm{Mindaugas}\sur{Venckevičius}}\email{mindaugas.venckevicius@mif.stud.vu.lt}

\affil*[1]{\orgdiv{Institute of Mathematics}, \orgname{Vilnius university}, \orgaddress{\street{Naugarduko 24}, \city{Vilnius}, \postcode{LT-03225}, \country{Lithuania}}}

\abstract{The mathematical essence in life insurance spins around the search of the nu\-me\-ri\-cal characteristics of the random variables $T_x$, $\nu^{T_x}$, $T_x\nu^{T_x}$, etc., where $\nu$ (deterministic) denotes the discount multiplier and $T_x$ (random) is the future lifetime of an in\-di\-vi\-dual being of $x\in\{0,\,1,\,\ldots\}$ years old. This work provides some historical facts about T. Wittstein and G. Balducci and their mortality assumption. We also develop some formulas that make it easier to compute the moments of the mentioned random variables assuming that the survival function is interpolated according to Balducci's assumption. Derived formulas are verified using some hypothetical mortality data.
}




\maketitle

\section{Introduction}\label{sec:intr}

Let $X$ be the absolutely continuous random variable determining a person's lifetime. By $x\in\{0,\,1,\,\ldots\}=:\mathbb{N}_0$ we denote the integer age of a certain person. Let $s(u):=\mathbb{P}(X\geqslant u),\,u\geqslant 0$ denote the survival function. In life insurance, based on some mortality table, the exact values of the survival function in many instances are given over the integers only, i.e. $s(x),\,x\in\mathbb{N}_0$, and the problems that life insurance deals with, ask to compute certain numerical cha\-rac\-te\-ris\-tics of $X$ under the certain interpolation of $s(x)$. More precisely, given the fixed age $x\in\mathbb{N}_0$, we shall connect the value $s(x+k)$ with $s(x+k+1)$ when $k$ varies over $\mathbb{N}_0$, and characterize some random function of $X$. Perhaps, the most common and simple interpolation of $s(x),\,x\in\mathbb{N}_0$ is the assumption of {\bf uniform distribution of deads} (UDD):
\begin{align}\label{UDD}
s(x+k+t)=(1-t)s(x+k)+ts(x+k+1),\,x,\,k\in\mathbb{N}_0,\,t\in[0,\,1].
\end{align}
In other words, the UDD assumption \eqref{UDD} states that the survival function $s(u),\,u\geqslant 0$ is linear over the intervals $[0,\,1)$, $[1,\,2)$, $\ldots$ Notice that $x+k+t$ when $x,\,k\in\mathbb{N}_0$ and $t\in(0,\,1)$ defines the fractional age of a person.

    Another widely known assumption on $s(x),\,x\in\mathbb{N}_0$ interpolation is the so called {\bf Balducci's} assumption:
\begin{align}\label{Balducci}
\frac{1}{s(x+k+t)}=\frac{1-t}{s(x+k)}+\frac{t}{s(x+k+1)},\,x,\,k\in\mathbb{N}_0,\,t\in[0,\,1].
\end{align}
In comparison to the UDD assumption \eqref{UDD}, Balducci's assumption \eqref{Balducci} states that the function $1/s(u),\,u\geqslant 0$ (as long as $s(u)>0$) is linear over the intervals $[0,\,1)$, $[1,\,2)$, $\ldots$, see Figure \ref{f1}. 

\begin{figure}[H]
\centering
\includegraphics[scale=0.75]{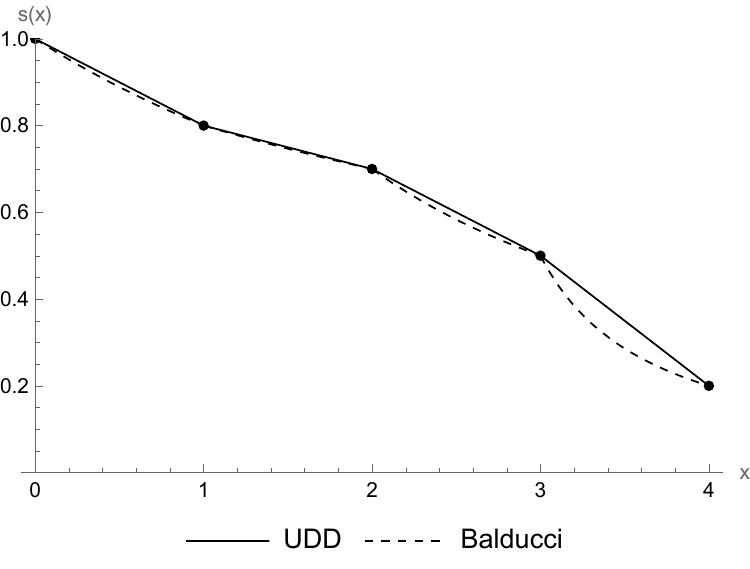}
\caption{View of $s(x)$ interpolation under UDD and Balducci's assumptions according to some hypothetical data.}
\label{f1}
\end{figure}

For $x\in\mathbb{N}_0$, we denote the future lifetime by $T_x$ of a person who is $x$ years old, i.e. $T_x=X-x$ given that $X\geqslant x$, where $X$ is the absolutely continuous random variable that determines a person's lifetime. Let $_{u}p_x$ denote the conditional survival function, i.e., given that $s(x)>0$,
\begin{align*}
_{u}p_x=\mathbb{P}(T_x\geqslant u)=\mathbb{P}(X\geqslant x+u|X\geqslant x)=\frac{s(x+u)}{s(x)},\,x\in\mathbb{N}_0,\,u\geqslant0.
\end{align*}
Then, by $f_x(u)$ we denote the conditional density, under condition $X\geqslant x$, of the random variable $X$ (person's lifetime) and define it as derivative
\begin{align}\label{cond_density}
f_x(u)=-\left( _{u}p_x\right)'_u.
\end{align}

Let us also denote $_{u}q_x:=1-_{u}p_x$, $_{1}p_x:=p_x$, and $_{1}q_x:=q_x$.
Under Balducci's assumption \eqref{Balducci}, the conditional density \eqref{cond_density}
is
\begin{align}\label{cond_dens_Bal}
f_x(k+t)=-(_{k+t}p_x)'_t=-\left(\frac{_{k+1}p_x}{p_{x+k}+t\cdot q_{x+k}}\right)'_t
=\frac{_{k+1}p_x\cdot q_{x+k}}{\left(1-(1-t)\cdot q_{x+k}\right)^2},
\end{align}
where $x,\,k\in\mathbb{N}_0$, and $t\in[0,\,1]$. In comparison, under UDD assumption \eqref{UDD}, the conditional density \eqref{cond_density}
is
\begin{align}\label{cond_dens_UDD}
f_x(k+t)=_{k}p_x-_{k+1}p_x=\frac{d_{x+k}}{l_x}=:_{k|1}q_x,
\end{align}
where $l_x$ represents the expected number of survivors to age $x$ from the $l_0$ newborns, i.e. $l_x=l_0s(x)$, and $d_x=l_x-l_{x+1}$ denote the number of deaths between ages $x$ and $x+1$. 
Let $g(u),\,u\geqslant 0$ be some real-valued function. In life insurance, to characterize the random variable $X$ essentially means to compute the expectation 
\begin{align}\label{expectation}
\mathbb{E}g(X)=\int\limits_{0}^{\infty}g(u)f_x(u)\,du=\int\limits_{0}^{1}g(u)f_x(u)\,du+
\int\limits_{1}^{2}g(u)f_x(u)\,du+\ldots
\end{align}
The computation of $\mathbb{E}g(X)$ in \eqref{expectation} is much easier using the density of UDD as the left-hand-side of \eqref{cond_dens_UDD} does not depend on $t$, and the same job becomes more complicated using the density \eqref{cond_dens_Bal}.

If the function $g$ in \eqref{expectation} represents the random discount multiplier according to the future lifetime, i.e. $g(X)=(1+i)^{-T_x}$, where $i>-1$ denotes the fixed annual interest rate, then the expectations of type \eqref{expectation} are called the {\bf net actuarial values}. As mentioned, the computation of the net actuarial values under the UDD assumption is more simplistic than the same under Balducci's assumption: under UDD we integrate $g$ over the "steps", while under Balducci's assumption, we are tasked to do the same over the ''arcs of hyperbolas'', see Figure \ref{f2}. 
\begin{figure}[H]
\centering
\includegraphics[scale=0.75]{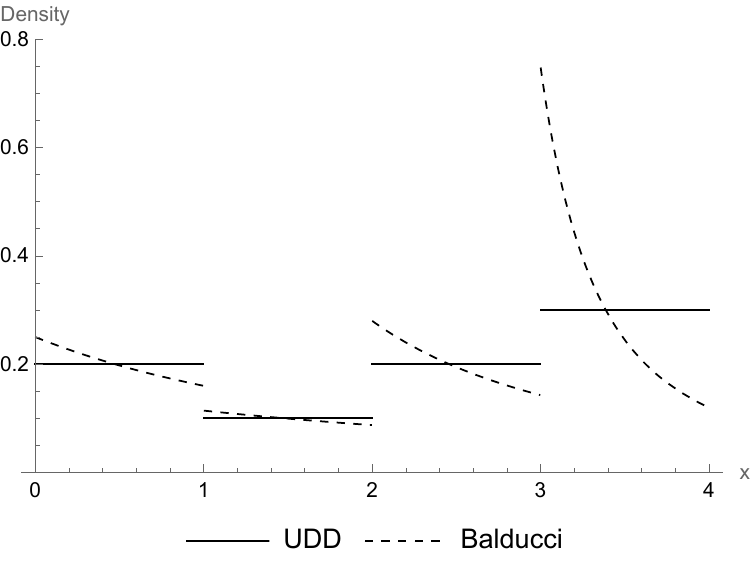}
\caption{View of the conditional density $f_x(k+t)$ under UDD and Balducci's assumptions according to some hypothetical data.}
\label{f2}
\end{figure}

Moreover, if the function $g$ in \eqref{expectation} is differentiable and non-increasing (non-decreasing), then the expected value of $g(T_x)$ under Balducci's mortality assumption \eqref{Balducci} is never less (greater) than $\mathbb{E} g(T_x)$ under the UDD assumption \eqref{UDD}, see Lemma \ref{lem:no_greater}. So, if $g(T_x)=(1+i)^{-T_x}$ then the net single premium under Balducci's assumption is never less than the net single premium under the UDD assumption. In practice, this leads to the higher expected payoff of the threatened claim and consequently the larger insurance price. 

It can be computed, see for instance \cite{batten1978mortality}, \cite{bowers1986actuarial}, that under Balducci's assumption, the force of mortality is
\begin{align}\label{FM_B}
\mu_{x+t}=-\frac{s'_t(x+t)}{s(x+t)}=\frac{q_x}{1-(1-t)\cdot q_x},\,x\in\mathbb{N}_0,\,0<t<1,
\end{align}
while the same under the UDD assumption is
\begin{align}\label{FM_UDD}
\mu_{x+t}=\frac{q_x}{1-t\cdot q_x},\,x\in\mathbb{N}_0,\,0<t<1.
\end{align}

Since the force of mortality \eqref{FM_B} decreases over the year, given that the one in \eqref{FM_UDD} increases, there are many insights and interpretations as to whether Balducci's mortality assumption is realistic for some human populations, see, for instance, \cite[p. 5]{batten1978mortality} cf. \cite[p. 105]{bowers1986actuarial}. The authors believe Balducci's assumption can be realistic for some human populations. For instance, newborns, teenagers who just reached the legal age to drive, consume alcohol, etc., or those individuals who, let's say, periodically face something that increases their death probability.

It is curious about the circumstances of how Balducci's assumption arose. Unfortunately, most publicly available sources, especially the electronic ones in English, lack even the basic facts about the life of Italian actuary Gaetano Balducci (1887 – 1974), see \cite[p. 148]{Balducci_years} where these birth and death years are provided. Balducci did numerous prominent works and officiated in state institutions during his lifetime. For example, it is most famously known that G. Balducci was a State general Accountant for Italy for almost ten years during the middle of the 20th century, see \cite[p. 7]{mosca_artbilancio}. Notably, he was also an active member of the Society of Actuaries (the global professional organization for actuaries), sharing his insights on life insurance theories in several issues of the Journal of Economists and Statistical Review, e.g. \cite{balducci1952}, \cite{balducci1911tavola}, \cite{balducci1917}.
Here we collect and interpret more facts about the assumption \eqref{Balducci} named after this Italian actuary. 

In fact, the interpolation \eqref{Balducci}, as well as UDD and constant mortality force, was first introduced by German mathematician Theodor Ludwig Wittstein (1816 – 1894), see \cite[p. 68]{bowers1986actuarial}. In contrast to Balducci, Wittstein's life appears to be much better documented in easily accessible public sources, see, for example, \cite[p. 358]{Theodor}. In short,  Wittstein was a high school teacher and textbook author, known for his contributions to the mathematical statistics field. In his work (1862) \cite{wittstein1862} Wittstein studied population mortality under certain assumptions and concluded that the mortality probability for the fractional ages can be computed as 
\begin{align*}
_{1-t}q_{x+t}=\frac{q_x \cdot(1-t)}{1-t\cdot q_x},\, x\in\mathbb{N}_0,\,0<t<1,
\end{align*}
what is equivalent to \eqref{Balducci}. In 1917, in work \cite{balducci1917} Balducci studied the development of certain populations too and concluded the same as Wittstein earlier. In our opinion, Balducci was unaware of Wittstein's work due to the different languages and the limited spread of information at the time. See the source \cite{puzey1993} and references therein for more facts and language unification regarding the studies by Wittstein and Balducci.

Let us mention that the formulas of the net actuarial values under UDD are well known, they are derived in many sources, see, for example, \cite{bowers1986actuarial}. Also, the constant mortality force implies the future lifetime $T_x$ (regardless of person's age) being distributed exponentially, i.e. $f_x(t)=\lambda e^{-\lambda t}$, $\lambda>0$, $t\geqslant0$. So, the numerical characteristics of $T_x$ under the constant mortality force are nothing but characteristics of the exponential distribution. In this work, we derive some formulas that make it easier to compute the net actuarial values under Balducci's assumption. We show that the described computation, dependently on the type of insurance and other circumstances, is essentially based on the exponential integral \eqref{exp_int} whose values can be computed by many software. 

The main results of this work are listed in Section \ref{sec:results}. In Propositions \ref{prop:main}-\ref{prop:[T]_nu} we compute the $m$-th moments of the random variables $\nu^{T_x}$, $T_x$, $T_x\cdot \nu^{T_x}$, $[T_x+1]\cdot\nu^{T_x}$ respectively, where $[\cdot]$ denotes the integer part function. The provided moments are computed over the yearly intervals: $l\leqslant T_x<l+1$, $l+1\leqslant T_x<l+2$, $\ldots$, $l+n-1\leqslant T_x<l+n$, where $n\in\mathbb{N}$, $l\in\mathbb{N}_0$, $l$ provides the years of deferment, and $n$ describes the maturity (in years) of insurance. As the random variables $\nu^{T_x}$ and $T_x$ were already introduced, we mention that the random variable $T_x\nu^{T_x}$ describes the present value of the uniformly increasing insurance, while $[T_x+1]\nu^{T_x}$ denotes the present value of yearly increasing insurance when the payoff is immediate after the insurer's death. Based on these random variables, there are various other modifications possible: $[T_x+n]\nu^{T_x}$, $[n-T_x]\nu^{T_x}$, etc.
In the next two statements, Propositions \ref{prop:j_times} and \ref{prop:j_[T]_nu_x}, we divide each year in $j\in\mathbb{N}$ equal pieces and compute the $m$-th moments of the random variables
\begin{align}\label{rvs_5_6}
\nu^{\left([T_x j]+1\right)/j}, \qquad [j T_x+1]\cdot\nu^{T_x},
\end{align}
where the first random variable in \eqref{rvs_5_6} can be used to express some net actuarial value when the installments are paid $j$ times per year, and the second random variable in \eqref{rvs_5_6} describes the present value of the payoff when the insurance amount increases $j$ times per year. Authors anticipate that the studied expectations of the provided random variables cover the most relevant insurance types or can be easily modified in a desired way.  

In Section \ref{sec:examples} we provide some numerical outputs of Propositions \ref{prop:main}-\ref{prop:j_[T]_nu_x} based on hypothetical mortality data. The provided examples were created by the software \cite{Mathematica} and their outputs were double-verified to match the results using the direct sum-integral method as provided in \eqref{expectation}.  

\section{Main results}\label{sec:results}

Let us denote the exponent integral
\begin{align}\label{exp_int}
\mathrm{Ei}(y):=-\int\limits_{-y}^{\infty}\frac{e^{-z}}{z}\,dz=\int\limits_{-\infty}^{y}\frac{e^z}{z}\,dz,\,y\in\mathbb{R}\setminus\{0\}
\end{align}
and recall two more standard notations in actuarial mathematics:
\begin{align*}
\delta:=\log(1+i),\,\nu:=\frac{1}{1+i},\,i> -1,
\end{align*}
where $i$ denotes the annual return rate, see \cite{notations} for more international actuarial notations used in the further text. We start with the statement on the $m$-th moment of the random variable $\nu^{T_x}$.  

\bigskip

\begin{proposition}\label{prop:main}
Say that the survival function $s(x)$, $x\in\mathbb{N}_0$ is interpolated according to Balducci's assumption \eqref{Balducci} and let $T_x$ denote the future lifetime of a person being of $x\in\mathbb{N}_0$ years old. If $p_{x+k}>0$ and $q_{x+k}>0$ for all $l\leqslant k \leqslant l+n-1$, where $l\in\mathbb{N}_0$ and $n\in\mathbb{N}$, then
\begin{align}
^m_{l|}\bar{A}_{x:\actuarialangle{n}}^1&:=\mathbb{E}\nu^{mT_x}\mathbbm{1}_{\{l\leqslant T_x < l+n\}}\nonumber 
=\nu^{ml}\cdot _lp_x-\nu^{m(l+n)}\cdot _{l+n}p_x\\
&+m\delta\sum_{k=l}^{l+n-1}\nu^{m(1+k-1/q_{x+k})}\cdot\frac{{_{k+1}p}_x}{q_{x+k}}\cdot \mathrm{Ei}_{k}(m\delta),\label{for_n_years_l_delay}
\\
^m_{l|}\bar{A}_{x}&:=\lim_{n\to\infty}\, ^m_{l|}\bar{A}_{x:\actuarialangle{n}}^1,
\label{till_the_end_of_life_l}
\end{align}
where $m\in\mathbb{N}$,
\begin{align}\label{eq:exp_dif}
\mathrm{Ei}_{k}(\delta)=\mathrm{Ei}\left(-\frac{\delta\cdot p_{x+k}}{q_{x+k}}\right)
-\mathrm{Ei}\left(-\frac{\delta}{q_{x+k}}\right)
\end{align}
and $\mathrm{Ei}(\cdot)$ denotes the exponent integral \eqref{exp_int}.
\end{proposition}

\bigskip

{\sc Remark 1:} {\it computing $^m_{l|}\bar{A}_{x}$ in \eqref{till_the_end_of_life_l} we shall let $n\in\mathbb{N}$ run up to such a number that $s(x+n-1+l)>0$ due to $q_{x+k+l}=1-s(x+k+l+1)/s(x+k+l)$ which is valid as long as $s(x+k+l)>0$. In other words, if $p_{x+k}=0$ for some $x,\,k\in\mathbb{N}_0$ in sum \eqref{for_n_years_l_delay}, then these corresponding summands are zeros
due to $_{k+1}p_x=_{k}p_x\cdot p_{x+k}$ and
\begin{align*}
\lim_{p_{x+k}\to 0}\,p_{x+k}\cdot \mathrm{Ei}\left(-\frac{\delta\cdot p_{x+k}}{q_{x+k}}\right)=0.
\end{align*}
Of course, $p_{x+k}=0$ implies that the survival function $s(y)=0$ for all $y\geqslant x+k$. In addition, in \eqref{for_n_years_l_delay} it can be shown that $\nu^{m(l+n)}\cdot _{l+n}p_x\to0$ as $n\to\infty$ if $\mathbb{E}\nu^{mT_x}$ exists.
} 

\bigskip

{\sc Remark 2:} {\it if $q_{x+k}=0$ for some $x,\,k\in\mathbb{N}_0$ in sum \eqref{for_n_years_l_delay}, then these corresponding summands, including the multiplication by $m\delta$, shall be understood as
\begin{align*}
\nu^{m k}\cdot (\nu^m-1)\cdot _{k}p_x
\end{align*}
due to
\begin{align*}
\lim_{\Delta\to 0}\frac{m\delta}{\Delta}\exp\left\{\frac{m\delta}{\Delta}\right\}\left(\mathrm{Ei}\left(-\frac{m\delta}{\Delta}+m\delta\right)-\mathrm{Ei}\left(-\frac{m\delta}{\Delta}\right)\right)=1-\nu^{-m},\,\Delta:=q_{x+k}.
\end{align*}
}

Let us mention that the values of the exponential integral \eqref{exp_int} can be computed according to the formula
$\mathrm{Ei}(-z)=-\mathrm{E}_1(z),\,z>0$, where
\begin{align}\label{approx}
\mathrm{E}_1(z)=-\gamma-\log z-\sum_{k=1}^{\infty}\frac{(-1)^k z^k}{k\cdot k!},
\end{align}
$\gamma=0.577\ldots$ is the Euler (also known as Euler-Mascheroni) constant. See \cite{exp_1}, \cite{exp_2}, and \cite{exp_3} for the various approximations the exponent integral.

The next proposition provides some numerical characteristics of the future lifetime $T_x$ when the survival function over the fractional ages is interpolated as in \eqref{Balducci}. In this case, the sum of hypergeometric functions \cite{Specia_F} can give the general expression of the $m$-th moment. However, this provides little value in this context, and we explicitly write down just a few of the first moments. In all subsequent propositions, $n\to\infty$ is allowed under the same means as in Proposition \ref{prop:main} and description in {\sc Remark 1}.

\bigskip

\begin{proposition}\label{prop:T_x}
Say that the survival function $s(x)$, $x\in\mathbb{N}_0$ is interpolated according to Balducci's assumption \eqref{Balducci} and let $T_x$ denote the future lifetime of a person being of $x\in\mathbb{N}_0$ years old. If $q_{x+k}<1$ for all $l\leqslant k \leqslant l+n-1$, where $l\in\mathbb{N}_0$, $n\in\mathbb{N}$, then
\begin{align}\label{T_x_m}\nonumber
\mathbb{E}\left(T_x\right)^m\mathbbm{1}_{\{l\leqslant T_x< n+l\}}
&=
_lp_x\cdot l^m-_{l+n}p_x\cdot(l+n)^m\\
&+m\sum_{k=l}^{l+n-1}\int_{k}^{k+1}\frac{_{k+1}p_x\cdot t^{m-1}dt}{1-(k+1-t)\cdot q_{x+k}},
\end{align}
where $m\in\mathbb{N}$.

In particular, if, in addition, $q_{x+k}>0$ for all $l\leqslant k \leqslant l+n-1$, and $l_x=l_0\cdot s(x),\,l_0\in\mathbb{N}$, then
\begin{align}
\mathbb{E}\mathbbm{1}_{\{l\leqslant T_x< n+l\}}&=\sum_{k=l}^{l+n-1} {_k}p_x\cdot q_{x+k}=\frac{1}{l_x}\sum_{k=l}^{l+n-1}d_{x+k}=\frac{l_{x+l}-l_{x+l+n}}{l_x},\label{0}\\
{{_{l|}}\mathop{e}\limits^\circ}_{x:\actuarialangle{n}}:=\mathbb{E}T_x\mathbbm{1}_{\{l\leqslant T_x< n+l\}}&=l\cdot {_l}p_x-(n+l)\cdot{_{n+l}p_x}-\sum_{k=l}^{n+l-1}\frac{_{k+1}p_x\cdot\log p_{x+k}}{q_{x+k}},\label{1}\\
\mathbb{E}T_x^2\mathbbm{1}_{\{l\leqslant T_x< n+l\}}&=
l^2\cdot {_l}p_x-(n+l)^2\cdot{_{n+l}p_x}\nonumber\\
&\hspace{-0.5cm}+2\sum_{k=l}^{n+l-1}\frac{_{k+1}p_x\cdot(q_{x+k}+(p_{x+k}-k\cdot q_{x+k})\cdot \log p_{x+k})}{q_{x+k}^2}.
\label{2}
\end{align}
\end{proposition}

{\sc Remark 3:} {\it Formula \eqref{0} does not depend on the interpolation of $s(x)$ over fractional age. The expression \eqref{1} in a different form with $l=0$ and $n\to\infty$ is also given in \cite[p. 24]{bravo2007tabuas}. It can be shown that $_{l+n}p_x\cdot(l+n)^m\to 0$ if $n\to\infty$ and $\mathbb{E}(T_x)^m$, $m\in\mathbb{N}$ exists. The summands in \eqref{1} and \eqref{2} \textup{(also \eqref{T_x_m})} are zeros if $p_{x+k}=0$ for some $x,\,k\in\mathbb{N}_0$, and 
\begin{align*}
\frac{\log p_{x+k}}{q_{x+k}}\to-1,\quad
\frac{q_{x+k}+(p_{x+k}-k\cdot q_{x+k})\cdot \log p_{x+k}}{q_{x+k}^2}\to\frac{1}{2}+k
\end{align*}
as $q_{x+k}\to0$ for some $x,\,k\in\mathbb{N}_0$. See {\sc Remark 1} under Proposition \ref{prop:main} also.
} 

\bigskip

    In the next Proposition, we compute the $m$-th moment of the random variable $T_x\cdot \nu^{T_x}$, which describes the present value of uniformly increasing payoff. Again, the general expression is complicated and we write down explicitly just the first two moments. 

\bigskip

\begin{proposition}\label{increasing}
Say that the survival function $s(x)$, $x\in\mathbb{N}_0$ is interpolated according to Balducci's assumption \eqref{Balducci} and let $T_x$ denote the future lifetime of a person who is $x\in\mathbb{N}_0$ years old. If $q_{x+k}<1$ for all $l\leqslant k \leqslant l+n-1$, where $l\in\mathbb{N}_0$, $n\in\mathbb{N}$, then
\begin{align}\nonumber
_{l|}^m\left(\bar{I}\bar{A}\right)_{x:\actuarialangle{n}}^1:=\mathbb{E}\,(T_x\nu^{T_x})^m\mathbbm{1}_{\{l\leqslant T_x< n+l\}}
&=l^m\cdot\nu^{ml}\cdot _lp_x-(l+n)^m\cdot \nu^{m(l+n)}\cdot _{l+n}p_x\\
&\hspace{-1cm}+m\sum_{k=l}^{l+n-1} {_{k+1}}p_x \int_{k}^{k+1}\frac{t^{m-1}\cdot \nu^{mt}\cdot(1-\delta t)}{1-(k+1-t)q_{x+k}}\,dt,\label{increasing_gen}
\end{align}
where $m\in\mathbb{N}$.

    In particular, if, in addition, $p_{x+k}>0$ and $q_{x+k}>0$ for all $l\leqslant k \leqslant l+n-1$, then
\begin{align}\nonumber
_{l|}\left(\bar{I}\bar{A}\right)_{x:\actuarialangle{n}}^1&=\mathbb{E}\,T_x\nu^{T_x}\mathbbm{1}_{\{l\leqslant T_x< n+l\}}\\\nonumber
&=_lp_x\cdot l \cdot \nu^l-_{l+n}p_x\cdot (l+n) \cdot \nu^{l+n}-i\sum_{k=l}^{l+n-1}\frac{_{k+1}p_x}{q_{x+k}}\cdot \nu^{k+1}\\
&-\sum_{k=l}^{l+n-1}\left(1-\delta\cdot\left(1+k-\frac{1}{q_{x+k}}\right)\right)\cdot
\frac{_{k+1}p_x}{q_{x+k}}\cdot \nu^{1+k-\frac{1}{q_{x+k}}}\cdot \mathrm{Ei}_k(\delta),\label{increasing_1}
\end{align}
where
\begin{align*}
\mathrm{Ei}_{k}(\delta)=\mathrm{Ei}\left(-\frac{\delta\cdot p_{x+k}}{q_{x+k}}\right)
-\mathrm{Ei}\left(-\frac{\delta}{q_{x+k}}\right)
\end{align*}
and $\mathrm{Ei}(\cdot)$ denotes the exponent integral \eqref{exp_int}.

Moreover, under the same assumptions,
\begin{align}\nonumber
^2_{l|}\left(\bar{I}\bar{A}\right)_{x:\actuarialangle{n}}^1&=\mathbb{E}\left(T_x\nu^{T_x}\right)^2\mathbbm{1}_{\{l\leqslant T_x< n+l\}}=\\
&l^2\cdot \nu^{2l}\cdot _lp_x-(l+n)^2\cdot\nu^{2(l+n)}\cdot _{l+n}p_x+
\sum_{k=l}^{n+l-1}\, _{k+1}p_x\cdot \left(I_{1,\,k}+I_{2,\,k}\right),\label{increasing_2}
\end{align}
where $I_{1,\,k}$ and $I_{2,\,k}$ are given in \eqref{I_1} and \eqref{I_2} respectively.
\end{proposition}

\bigskip

{\sc Remark 4: }{\it as previously, the summands in \eqref{increasing_1} and \eqref{increasing_2} are zeros if $p_{x+k}=0$ for some $x,\,k\in\mathbb{N}_0$. If $q_{x+k}=0$ for some $x,\,k\in\mathbb{N}_0$ in the corresponding terms of two sums of \eqref{increasing_1}, then the limit, as $q_{x+k}\to0$, is $\nu^{k+1}\cdot{_k}p_x\cdot(1-i\cdot k)$. 
If $q_{x+k}=0$ for some $x,\,k\in\mathbb{N}_0$ in some summands in \eqref{increasing_2}, then the limit, as $q_{x+k}\to0$, is $\nu^{2k+2}\cdot {_k}p_x\cdot\left(1+2k-k^2\cdot(i^2+2i)\right)$.
}

\bigskip

In the following proposition, we compute the $m$-th moment of the random variable $[T_x+1]\cdot\nu^{T_x}$ which describes the present value of the yearly increasing payoff.

\bigskip

\begin{proposition}\label{prop:[T]_nu}
Say that the survival function $s(x)$, $x\in\mathbb{N}_0$ is interpolated according to Balducci's assumption \eqref{Balducci} and let $T_x$ denote the future lifetime of a person being of $x\in\mathbb{N}_0$ years old. If $q_{x+k}>0$ and $p_{x+k}>0$ for all $l\leqslant k \leqslant l+n-1$, where $l\in\mathbb{N}_0$ and $n\in\mathbb{N}$, then
\begin{align}\label{eq:in_prop_4}\nonumber
&^m_{l|}\left(I\bar{A}\right)_{x:\actuarialangle{n}}^1:=
\mathbb{E}\,([T_x+1]\nu^{T_x})^m\mathbbm{1}_{\{l\leqslant T_x< n+l\}}\\\nonumber
&=\sum_{k=l}^{l+n-1} {_k}p_x\cdot (k+1)^m \cdot \nu^{mk}\cdot(1-p_{x+k}\cdot \nu^m)\\
&+\delta m \sum_{k=l}^{l+n-1} {_{k+1}}p_x\cdot (k+1)^m\cdot \frac{v^{m(1+k-1/q_{x+k})}}{q_{x+k}}\cdot \mathrm{Ei}_k(m\delta),
\end{align}
where $m\in\mathbb{N}$,
\begin{align*}
\mathrm{Ei}_{k}(\delta)=\mathrm{Ei}\left(-\frac{\delta\cdot p_{x+k}}{q_{x+k}}\right)
-\mathrm{Ei}\left(-\frac{\delta}{q_{x+k}}\right)
\end{align*}
and $\mathrm{Ei}(\cdot)$ denotes the exponent integral \eqref{exp_int}.
\end{proposition}

\bigskip

{\sc Remark 5: }{\it again, if $p_{x+k}=0$ for some $x,\,k\in\mathbb{N}_0$ in some summands in \eqref{eq:in_prop_4}, then they are zeros. If $q_{x+k}=0$ for some $x,\,k\in\mathbb{N}_0$ in some summands in \eqref{eq:in_prop_4}, then the limits in these corresponding summands, as $q_{x+k} \to 0$, are ${_k}p_x\cdot k^m \cdot \nu^{mk}\cdot (\nu^m-1)$, see also {\sc Remark 2} above.
}

\bigskip

    In this work, we do not develop any formulas for the present actuarial 
values as they admit the representations via the net single premiums computed in Proposition \ref{prop:main}, see \cite[Ch. 5]{bowers1986actuarial}. On the other hand, when the insurer pays a certain amount with a different intensity than yearly (as long as he/she is alive), we shall look for some more convenient formulas that convert the present value of such cash flows to the present values of yearly mortality data-based cash flows.    

    Let us consider the time slot from $0$ up to $n\in\mathbb{N}$ years and split every single year in $j\in\mathbb{N}$ intervals as follows:
\begin{align}\label{intervals}
\hspace{-0.5cm}
\begin{cases}
\text{1'th year:}\left[0,\,\frac{1}{j}\right),\,\left[\frac{1}{j},\,\frac{2}{j}\right),\,\ldots,\,
\left[1-\frac{1}{j},\,1\right),\\
\text{2'nd year:}\left[1,\,1+\frac{1}{j}\right),\,\left[1+\frac{1}{j},\,1+\frac{2}{j}\right),\,\ldots,\,
\left[2-\frac{1}{j},\,2\right),\\
\vdots\\
\text{n'th year:}\left[n-1,\,n-1+\frac{1}{j}\right),\,\left[n-1+\frac{1}{j},\,n-1+\frac{2}{j}\right),\,\ldots,\,
\left[n-\frac{1}{j},\,n\right).
\end{cases}
\end{align}

In the next proposition, we provide the formula to compute the $m$-th moment of the random variable $\nu^{([T_x j]+1)/j}$, $j\in\mathbb{N}$, where the future lifetime $T_x$ is distributed over the intervals in \eqref{intervals}. In this situation, insurance deferment can be described by ''$l*n_1$'', where $l\in\mathbb{N}_0$ provides years, and $n_1\in\{0,\,1,\,\ldots,\,j-1\}$ means the number of periods whose length is $1/j$. For instance, if $j=12$, then $l=1$ and $n_1=2$ describe the deferment of a year and two months.   

\bigskip

\begin{proposition}\label{prop:j_times}
Say that the survival function $s(x)$, $x\in\mathbb{N}_0$ is interpolated according to Balducci's assumption \eqref{Balducci} and let $T_x$ denote the future lifetime of a person being of $x\in\mathbb{N}_0$ years old. If $j,\,n\in\mathbb{N}$, $m,\,l\in\mathbb{N}_0$, and $n_1\in\{0,\,1,\,\ldots,\,j-1\}$, then
\begin{align}\label{prop5_eq1}\nonumber
&^m_{l|}\left(\bar{A}^{(j)}\right)_{x:\actuarialangle{n}}^1:=
\mathbb{E}\left(\nu^{\frac{[T_x\,j]+1}{j}}\right)^m\mathbbm{1}_{\{l*n_1\leqslant T_x<n+l*n_1\}}\\\nonumber
&=\frac{1}{j}\sum_{k=l}^{n+l-1}\sum_{d=n_1}^{j-1}\nu^{(k+(d+1)/j)m}\frac{_{k+1}p_x\cdot q_{x+k}}{\left(p_{x+k}+\frac{d}{j}\cdot q_{x+k}\right)\left(p_{x+k}+\frac{d+1}{j}\cdot q_{x+k}\right)}+\\
&\frac{\nu^{(n+l)m}}{j}{_{n+l+1}}p_x\cdot q_{x+n+l} \sum_{d=0}^{n_1-1}
\frac{\nu^{(d+1)/j\cdot m}}{\left(p_{x+n+l}+\frac{d}{j}\cdot q_{x+n+l}\right)
\left(p_{x+n+l}+\frac{d+1}{j}\cdot q_{x+n+l}\right)
}
.
\end{align}
\end{proposition}

In the last proposition, we provide the formula to compute the $m$-th moment of the random variable $[j\cdot T_x+1]\cdot\nu^{T_x}$, $j\in\mathbb{N}$, where the future lifetime $T_x$ again is distributed over the intervals in \eqref{intervals}.

\bigskip

\begin{proposition}\label{prop:j_[T]_nu_x}
Say that the survival function $s(x)$, $x\in\mathbb{N}_0$ is interpolated according to Balducci's assumption \eqref{Balducci} and let $T_x$ denote the future lifetime of a person being of $x\in\mathbb{N}_0$ years old. If $j,\,n\in\mathbb{N}$, $m,\,l\in\mathbb{N}_0$, $n_1\in\{0,\,1,\,\ldots,\,j-1\}$, $p_{x+k}>0$, and $q_{x+k}>0$ for all $l\leqslant k \leqslant n+l-1$, then
\begin{align*}
&^m_{l|}\left(I^{(j)}\bar{A}\right)_{x:\actuarialangle{n}}^1:=
\mathbb{E}\left([j\cdot T_x+1]\cdot \nu^{T_x}\right)^m\mathbbm{1}_{\{l*n_1\leqslant T_x < n+l*n_1\}}=\\
&
\sum_{k=l}^{n+l-1}\sum_{d=n_1}^{j-1}(d+jk+1)^m\cdot {_{k+1}}p_x\cdot\nu^{m\left(k+\frac{d}{j}\right)}\left(\frac{1}{1-\left(1-\frac{d}{j}\right)\cdot q_{x+k}}-\frac{\nu^{m/j}}{1-\left(1-\frac{d+1}{j}\right)\cdot q_{x+k}}\right)\\
&+{_{n+l+1}}p_x\cdot \nu^{m(n+l)}\times\\
&\sum_{d=0}^{n_1-1}(d+j\cdot(n+l)+1)^m\cdot\nu^{\frac{md}{j}}
\left(\frac{1}{1-\left(1-\frac{d}{j}\right)\cdot q_{x+n+l}}-\frac{\nu^{m/j}}{1-\left(1-\frac{d+1}{j}\right)\cdot q_{x+n+l}}\right)
\\
&-\delta m\sum_{k=l}^{n+l-1}\sum_{d=n_1}^{j-1}(d+j\cdot k+1)^m\cdot {_{k+1}p_x}\cdot \frac{\nu^{m(1+k-1/q_{x+k})}}{q_{x+k}}\cdot \mathrm{Ei}_k(m\delta,\,d,\,j)\\
&-
\delta m \cdot {_{n+l+1}}p_x\cdot \frac{\nu^{m(1+n+l-1/q_{x+k})}}{q_{x+n+l}}\sum_{d=n_1}^{j-1}(d+j\cdot (n+l)+1)^m\cdot \mathrm{Ei}_{n+l}(m\delta,\,d,\,j)
.
\end{align*}
Here
\begin{align*}
\mathrm{Ei}_{k}(\delta,\,d,\,j)=\mathrm{Ei}\left(-\delta\left(\frac{p_{x+k}}{q_{x+k}}+\frac{d+1}{j}\right)\right)
-
\mathrm{Ei}\left(-\delta\left(\frac{p_{x+k}}{q_{x+k}}+\frac{d}{j}\right)\right)
\end{align*}
and $\mathrm{Ei}(\cdot)$ denotes the exponent integral \eqref{exp_int}.
\end{proposition}

\bigskip

{\sc Remark 6:} as previously, $p_{x+k}=0$ implies zero summands in Proposition \ref{prop:j_[T]_nu_x}. If $q_{x+k}=0$, then
\begin{align*}
&\lim_{\Delta\to 0}\frac{m\delta}{\Delta}\exp\left\{\frac{m\delta}{\Delta}\right\}\left(\mathrm{Ei}\left(-m\delta\left(\frac{1-\Delta}{\Delta}+\frac{d+1}{j}\right)\right)-\mathrm{Ei}\left(-m\delta\left(\frac{1-\Delta}{\Delta}+\frac{d}{j}\right)\right)\right)\\
&=\nu^{m(d/j-1)}\cdot(1-\nu^{m/j}),\,\Delta:=q_{x+k}.
\end{align*}

As mentioned in the Introduction, the provided formulas in Propositions \ref{prop:main}-\ref{prop:j_[T]_nu_x} can be modified in various desired ways. For example, because of the Proposition \ref{prop:j_times} and under its assumptions 
\begin{align*}
&\mathbb{E}\left(\frac{[T_x\,j]+1}{j}\right)^m\mathbbm{1}_{\{l*0\leqslant T_x<n+l*0\}}\\
&=\frac{1}{j}\sum_{k=l}^{n+l-1}\sum_{d=0}^{j-1}\left(k+\frac{d+1}{j}\right)^m\frac{_{k+1}p_x\cdot q_{x+k}}{\left(p_{x+k}+\frac{d}{j}\cdot q_{x+k}\right)\left(p_{x+k}+\frac{d+1}{j}\cdot q_{x+k}\right)}.
\end{align*}

\section{Some illustrative examples}\label{sec:examples}

In this section, we give two examples that verify Propositions \ref{prop:main}--\ref{prop:j_[T]_nu_x} based on some hypothetical survival laws. As mentioned in the Introduction, the obtained outputs are double-verified according to the formula \eqref{expectation}. These examples serve the purpose of convincing the correctness of the given computational formulas rather than reflecting on something from real life, see \cite{Brazilian}, \cite{AMJ}, \cite{CMSTM}, references therein, and many other sources for studying the survival laws of human populations. On the other hand, papers such as \cite{RSJ}, deal with some theoretical studies regarding the survival (tail) function. 

\bigskip

\begin{ex}
Let $l=0$, $i=5\%$, $n=10$, and
\begin{align*}
_kp_x=\frac{100-k}{100},\,k=1,\,2,\,\ldots,\,10.
\end{align*}
We compute the expectations given in Propositions \ref{prop:main}--\ref{prop:j_[T]_nu_x} when $m=1$ or $m=2$.
\end{ex}

\bigskip

The provided conditional survival function $_kp_x$ implies
\begin{align*}
q_{x+k}&=\frac{1}{100-k},\,k=0,\,1,\,\ldots,\,9.
\end{align*}

The formula \eqref{for_n_years_l_delay} in Proposition \ref{prop:main} yields:
\begin{align*}
\bar{A}^1_{x:\actuarialangle{10}}
\approx0.0791388,\qquad
^2\bar{A}^1_{x:\actuarialangle{10}}\approx0.063867.
\end{align*}
According to Proposition \ref{prop:T_x}:
\begin{align*}
\frac{l_x-l_{x+10}}{l_x}=\frac{1}{10},\qquad
{\mathop{e}\limits^\circ}_{x:\actuarialangle{10}}\approx0.499824,\qquad
\mathbb{E}T_x^2\mathbbm{1}_{\{T_x<10\}}\approx3.33155.
\end{align*}

Proposition \ref{increasing} gives:
\begin{align*}
\left(\bar{I}\bar{A}\right)_{x:\actuarialangle{10}}^1\approx0.363507,\qquad
^2\left(\bar{I}\bar{A}\right)_{x:\actuarialangle{10}}^1\approx1.63319.
\end{align*}

Proposition \ref{prop:[T]_nu} yields:
\begin{align*}
\left(I\bar{A}\right)_{x:\actuarialangle{10}}^1\approx0.403536,\qquad
^2\left(I\bar{A}\right)_{x:\actuarialangle{10}}^1\approx1.91788.
\end{align*}

If, in addition, $j=2$, then Propositions \ref{prop:j_times} and \ref{prop:j_[T]_nu_x} respectively give: 
\begin{align*}
&\left(\bar{A}^{(2)}\right)_{x:\actuarialangle{10}}^1\approx 0.0781758,\quad
^2\left(\bar{A}^{(2)}\right)_{x:\actuarialangle{10}}^1\approx 0.062319,\\
&\left(I^{(2)}\bar{A}\right)_{x:\actuarialangle{10}}^1\approx 0.766813,\quad
^2\left(I^{(2)}\bar{A}\right)_{x:\actuarialangle{10}}^1\approx7.08521.
\end{align*}

\begin{ex}
Let $l=1$, $x=0$, $i=5\%$, $n\to\infty$, and
\begin{align}\label{Weibul}
{_k}p_x=\exp\left\{\left(\frac{x}{\alpha}\right)^\beta-\left(\frac{k+x}{\alpha}\right)^\beta\right\},\,k\in\mathbb{N}_0,\,\alpha=50,\,\beta=3.
\end{align}
We compute the expectations given in Propositions \ref{prop:main}--\ref{prop:j_[T]_nu_x} when $m=1$ or $m=2$.
\end{ex}

\bigskip

The conditional probability distribution in \eqref{Weibul} is known as a discrete Weibull distribution with positive parameters $\alpha$ and $\beta$, see Figures \ref{fig:1}, \ref{fig:2}, and sources \cite{Weibull}, \cite{Weibull_1}, on some recent studies regarding this distribution.
\begin{figure}[H]
    \centering
    \begin{minipage}{0.5\textwidth}\
        \centering
        \includegraphics[width=0.9\textwidth]{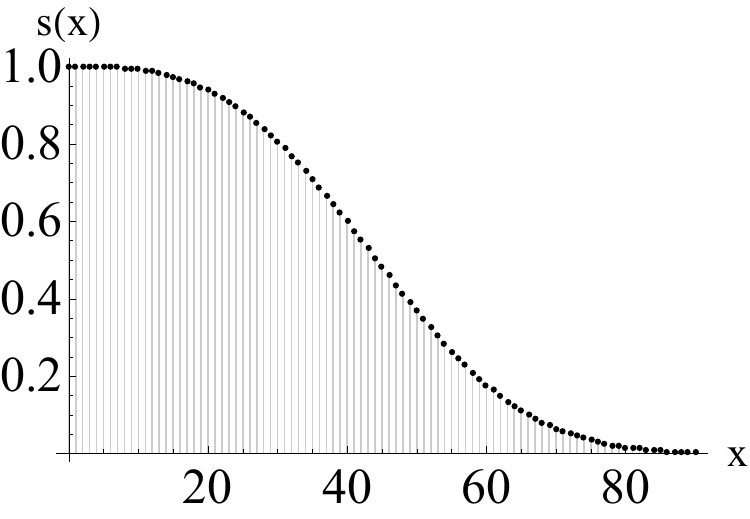} 
        \caption{The survival function $s(x)=$\\ $\exp\{-(x/\alpha)^\beta\}$, $x\in\mathbb{N}_0$, $\alpha=50$, $\beta=3$.}\label{fig:1}
    \end{minipage}\hfill
    \begin{minipage}{0.5\textwidth}
        \centering
        \includegraphics[width=0.9\textwidth]{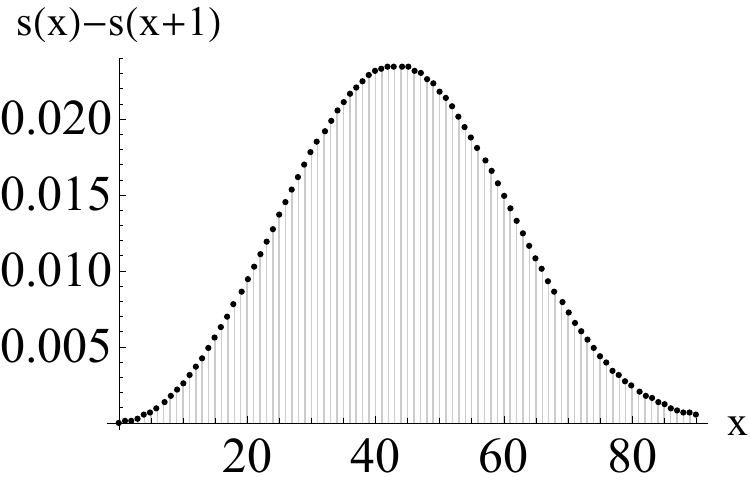}
        \caption{The probability mass function $\mathbb{P}(X=x)=s(x)-s(x+1)$, $x\in\mathbb{N}_0$.}
        \label{fig:2}
    \end{minipage}
\end{figure}
The provided mortality law implies
\begin{align*}
p_{x+k}=\exp\left\{\left(\frac{k+x}{\alpha}\right)^\beta-\left(\frac{k+1+x}{\alpha}\right)^\beta\right\},\,x,\,k\in\mathbb{N}_0,\,\alpha=50,\,\beta=3.
\end{align*}

Thus, according to Proposition \ref{prop:main}:
\begin{align*}
_{1|}\bar{A}_{0}
\approx0.152212,\qquad
_{1|}^2\bar{A}_{0}\approx0.0381506.
\end{align*}

Proposition \ref{prop:T_x}:
\begin{align*}
\lim_{n\to\infty}\frac{l_{1}-l_{1+n}}{l_0}=e^{-0.000008}\approx0.999992,\quad
_{1|}{\mathop{e}\limits^\circ}_0\approx44.6399,\quad
\mathbb{E}T_0^2\mathbbm{1}_{\{1\leqslant T_0\}}\approx2256.03.
\end{align*}

Proposition \ref{increasing}:
\begin{align*}
_{1|}\left(\bar{I}\bar{A}\right)_{0}\approx5.01701,\qquad
_{1|}^2\left(\bar{I}\bar{A}\right)_{0}\approx28.0812.
\end{align*}

Proposition \ref{prop:[T]_nu}:
\begin{align*}
_{1|}\left(I\bar{A}\right)_{0}\approx5.09453,\qquad
_{1|}^2\left(I\bar{A}\right)_{0}\approx29.0377.
\end{align*}

If, in addition, $j=12$ and $n_1=2$ (there is a deferment of a year and two months), then Propositions \ref{prop:j_times} and \ref{prop:j_[T]_nu_x} respectively give: 
\begin{align*}
&_{1*2|}\left(\bar{A}^{(12)}\right)_{0}\approx0.125392 ,\quad
_{1*2|}^2\left(\bar{A}^{(12)}\right)_{0}\approx0.0313064,\\
&_{1*2|}\left(I^{(12)}\bar{A}\right)_{0}\approx10.1107 ,\quad
_{1*2|}^2\left(I^{(12)}\bar{A}\right)_{0}\approx114.207.
\end{align*}

\section{Lemmas}
In this section, we formulate and prove three auxiliary statements. Let us recall that
\begin{align*}
\delta=\log(1+i),\,\nu=\frac{1}{1+i},\,i>-1.
\end{align*}
and 
\begin{align*}
\mathrm{Ei}(y)=-\int\limits_{-y}^{\infty}\frac{e^{-z}}{z}\,dz=\int\limits_{-\infty}^{y}\frac{e^z}{z}\,dz,\,y\in\mathbb{R}\setminus\{0\}.
\end{align*}

\begin{lem}\label{lem:spec_int}
Let $x,\,k\in\mathbb{N}_0$, $m,\,j\in\mathbb{N}$, $d\in\{0,\,1,\,\ldots,\,j-1\}$, and suppose that $p_{x+k}>0$ and $q_{x+k}>0$. Then
\begin{align}
\int\limits_{k+d/j}^{k+(d+1)/j}\frac{\nu^{m\,t}\,dt}{1-(1+k-t)\cdot q_{x+k}}=\frac{\nu^{m(1+k-1/q_{x+k})}}{q_{x+k}}\mathrm{Ei}_k(m\delta,\,d,\,j),
\end{align}
where
\begin{align*}
\mathrm{Ei}_k(\delta,\,d,\,j)=\mathrm{Ei}\left(-\delta\left(\frac{p_{x+k}}{q_{x+k}}+\frac{d+1}{j}\right)\right)
-
\mathrm{Ei}\left(-\delta\left(\frac{p_{x+k}}{q_{x+k}}+\frac{d}{j}\right)\right).
\end{align*}
\end{lem}

\begin{proof}
The proof is straightforward
\begin{align*}
&\int\limits_{k+d/j}^{k+(d+1)/j}\frac{\nu^{m\,t}\,dt}{1-(1+k-t)\cdot q_{x+k}}\\
&=\frac{\nu^{m(1+k-1/q_{x+k})}}{q_{x+k}}
\int\limits_{p_{x+k}+d/j \cdot q_{x+k}}^{p_{x+k}+(d+1)/j \cdot q_{x+k}}
\frac{e^{-\delta\,m\,y/q_{x+k}}}{\delta\,m\,y/q_{x+k}}\,d\left(\frac{\delta \, m\,y}{q_{x+k}}\right)\\
&=\frac{\nu^{m(1+k-1/q_{x+k})}}{q_{x+k}}
\left(
\mathrm{Ei}\left(-m\delta\left(\frac{p_{x+k}}{q_{x+k}}+\frac{d+1}{j}\right)\right)
-
\mathrm{Ei}\left(-m\delta\left(\frac{p_{x+k}}{q_{x+k}}+\frac{d}{j}\right)\right)
\right).
\end{align*}
\end{proof}

\begin{lem}\label{lem:j_times}
Let $x,\,a\in\mathbb{N}_0$, $j\in\mathbb{N}$, and $d=0,\,1,\,\ldots,\,j-1$. Then, under Balducci's mortality assumption \eqref{Balducci},
\begin{align*}
_{a+\frac{d}{j}}p_x-_{a+\frac{d+1}{j}}p_x
=_{a+1}p_x\cdot q_{x+a}\cdot\frac{1}{j}\cdot\frac{1}{p_{x+a}+\frac{d}{j}\cdot q_{x+a}}\cdot\frac{1}{p_{x+a}+\frac{d+1}{j}\cdot q_{x+a}}.
\end{align*}
\end{lem}
\begin{proof}
Since $a+x\in\mathbb{N}_0$, $d/j\in[0,\,1)$, and $(d+1)/j\in(0,\,1]$, we then apply the interpolation \eqref{Balducci} and obtain
\begin{align*}
&s\left(x+a+\frac{d}{j}\right)-s\left(x+a+\frac{d+1}{j}\right)\\
&=\frac{s(x+a)s(x+a+1)}{\frac{d}{j}\cdot s(x+a)+\left(1-\frac{d}{j}\right)\cdot s(x+a+1)}
-\frac{s(x+a)s(x+a+1)}{\frac{d+1}{j}\cdot s(x+a)+\left(1-\frac{d+1}{j}\right)\cdot s(x+a+1)}\\
&=s(x)\cdot _{a+1}p_x\cdot q_{x+a}\cdot j\cdot
\frac{1}{d\cdot q_{x+a}+j\cdot p_{x+a}}\cdot 
\frac{1}{(d+1)\cdot q_{x+a}+j\cdot p_{x+a}}.
\end{align*}
\end{proof}

\begin{lem}\label{lem:no_greater}
Let $g(t)$ be the real, differentiable and non-increasing function over the intervals $t\in[0,\,1)\cup [1,\,2),\, \ldots$ Then the expected value of $g(T_x)$ under Balducci's mortality assumption is never less than $\mathbb{E} g(T_x)$ under the UDD assumption. 

    Conversely, suppose that the function $g(t)$ is non-decreasing under the same conditions. In that case, the expected value of $g(T_x)$ under Balducci's mortality assumption is never greater than $\mathbb{E} g(T_x)$ under the UDD assumption. 
\end{lem}
\begin{proof}
Let $k\in\mathbb{N}_0$. Then, for any $l_x>0$, $x\in\mathbb{N}_0$ and $k$ such that $q_{x+k}<1$, we have
\begin{align*}
{\mathfrak I}&:=\int_{k}^{k+1}\left(\frac{_{k+1}p_x\cdot q_{x+k}}{(1-(k+1-t)q_{x+k})^2}-\frac{d_{x+k}}{l_x}\right)g(t)\,dt\\
&=\int_{k}^{k+1}g(t)\,d\left(-\frac{_{k+1}p_x}{1-(k+1-t)q_{x+k}}-\frac{d_{x+k}}{l_x}t\right)
=-\int_{k}^{k+1}g(t)\,d \,z(t),
\end{align*}
where
\begin{align*}
z(t)=\frac{_{k+1}p_x}{1-(k+1-t)q_{x+k}}+\frac{d_{x+k}}{l_x}t\geqslant0,\,k\leqslant t< k+1.
\end{align*}
It is easy to check $z(k)=z(k+1)={_k}p_x+\frac{d_{x+k}}{l_x}\cdot k$. Then, upon the integration by parts,
\begin{align}\label{ineq:estimate}\nonumber
{\mathfrak I}&=-g(k+1)z(k+1)+g(k)z(k)+\int_{k}^{k+1}z(t)\,g'(t)\,dt\\
&=\left(g(k)-g(k+1)\right)\left(_{k}p_x+\frac{d_{x+k}}{l_x}\cdot k\right)
+\int_{k}^{k+1}z(t)\,g'(t)\,dt.
\end{align}
If $g'(t)\leqslant0$ in \eqref{ineq:estimate}, then 
\begin{align}\label{ineq:estimate_1}
{\mathfrak I}\geqslant\left(g(k)-g(k+1)\right)\left(_{k}p_x+\frac{d_{x+k}}{l_x}\cdot k\right)+\max_{0\leqslant t< 1}z(t) \cdot (g(k+1)-g(k)).
\end{align}
Since $z(k)=z(k+1)=_{k}p_x+\frac{d_{x+k}}{l_x}\cdot k$ and the second derivative
\begin{align*}
z''(t)=\frac{2\cdot _{k+1}p_x \cdot q_{x+k}^2}{(1-(k+1-t)\,q_{x+k})^3}\geqslant 0
\end{align*}
for all $k\leqslant t < k+1$, the function $z(t)$ is not concave down and consequently $\max_{0\leqslant t < 1} z(t)=_{k}p_x+\frac{d_{x+k}}{l_x}\cdot k$. Thus, ${\mathfrak I}\geqslant 0$ due to \eqref{ineq:estimate_1}.

Conversely, if $g'(t)\geqslant0$ in \eqref{ineq:estimate}, then 
\begin{align*}
{\mathfrak I}\leqslant\left(g(k)-g(k+1)\right)\left(_{k}p_x+\frac{d_{x+k}}{l_x}\cdot k\right)+\max_{0\leqslant t< 1}z(t) \cdot (g(k+1)-g(k))=0
\end{align*}
due to the same argumentation.
\end{proof}

\section{Propositions' proofs}
In this Section, we prove all propositions formulated in Section \ref{sec:results}. 
\begin{proof}[Proof of Proposition \ref{prop:main}]
Given the density under Balducci's assumption \eqref{cond_dens_Bal}, we obtain

\begin{align*}
_{l|}^m\bar{A}_{x:\actuarialangle{n}}^1&=\int\limits_{l}^{n+l}\nu^{m\,u}\,f_x(u)\,du
=\sum_{k=l}^{l+n-1}\int\limits_{k}^{k+1}\nu^{m\,t}\,f_x(t)\,dt\\
&=\sum_{k=l}^{l+n-1}
{_{k+1}p}_x\cdot q_{x+k}\int\limits_{k}^{k+1}\frac{\nu^{t\,m}\,dt}{(1-(k+1-t)\cdot q_{x+k})^2},\,n\in\mathbb{N}.
\end{align*}
According to Lemma \ref{lem:spec_int} with $j=1$  and $d=0$, the last integral, multiplied by $q_{x+k}\neq0$, is
\begin{align*}
&q_{x+k}\int\limits_{k}^{k+1}\frac{\nu^{m\, t}\,dt}{(1-(k+1-t)\cdot q_{x+k})^2}
=-\int\limits_{k}^{k+1}\nu^{m\, t}d\left(\frac{1}{1-(k+1-t)\cdot q_{x+k}}\right)\\
&=\frac{\nu^{m\, k}}{p_{x+k}}-\nu^{m(k+1)}-\delta m\int\limits_{k}^{k+1}
\frac{\nu^{m\, t}\,dt}{1-(k+1-t)\cdot q_{x+k}}\\
&=\frac{\nu^{m\, k}}{p_{x+k}}-\nu^{m(k+1)}
+\delta m\,\frac{\nu^{m(1+k-1/q_{x+k})}}{q_{x+k}}
\mathrm{Ei}_k(m\delta),
\end{align*}
where the difference of exponent integrals $\mathrm{Ei}_k(\delta)$ is given in \eqref{eq:exp_dif}. Then
\begin{align*}
&^m_{l|}\bar{A}_{x:\actuarialangle{n}}^1=\\
&\sum_{k=l}^{l+n-1}\nu^{m\,k}\cdot {_kp}_x-\sum_{k=l}^{l+n-1}\nu^{m(k+1)}\cdot {_{k+1}p}_x
+\delta m\sum_{k=l}^{l+n-1}\nu^{m(1+k-1/q_{x+k})}\cdot\frac{{_{k+1}p}_x}{q_{x+k}}
\cdot \mathrm{Ei}_k(m\delta)\\
&=\nu^{ml}\cdot _l p_x-\nu^{m(l+n)}\cdot _{l+n}p_x+\delta m\sum_{k=l}^{l+n-1}\nu^{m(1+k-1/q_{x+k})}\cdot\frac{{_{k+1}p}_x}{q_{x+k}}\cdot \mathrm{Ei}_k(m\delta).
\end{align*}
This proves the equality \eqref{for_n_years_l_delay} while \eqref{till_the_end_of_life_l} is implied by \eqref{for_n_years_l_delay} as $n\to\infty$.
\end{proof}

\begin{proof}[Proof of Proposition \ref{prop:T_x}]
Eq. \eqref{T_x_m} is implied due to
\begin{align*}
&\mathbb{E}\left(T_x\right)^m\mathbbm{1}_{\{l\leqslant T_x< n+l\}}=\sum_{k=l}^{l+n-1} {_{k+1}}p_x\cdot q_{x+k}\int\limits_{k}^{k+1}\frac{t^m\, dt}{\left(1-(k+1-t)\cdot q_{x+k}\right)^2},
\end{align*}
where
\begin{align*}
&q_{x+k}\int\limits_{k}^{k+1}\frac{t^m dt}{\left(1-(k+1-t)\cdot q_{x+k}\right)^2}\nonumber\\
&=\frac{k^m}{p_{x+k}}-(k+1)^m+m\int_{k}^{k+1}\frac{t^{m-1}dt}{1-(k+1-t)\cdot q_{x+k}},\,m\in\mathbb{N}.
\end{align*}

As $T_x$ is the continuous random variable, eq. \eqref{0} is straightforward due to
\begin{align*}
\mathbb{E}\mathbbm{1}_{\{l\leqslant T_x< n+l\}}
=\sum_{k=l}^{l+n-1}\mathbb{P}(k\leqslant T_x < k+1)
\end{align*}
and elementary rearrangements.

Eq. \eqref{1} is a corollary of \eqref{T_x_m} with $m=1$, where 
\begin{align*}
\sum_{k=l}^{l+n-1}\left(_kp_x\cdot k-_{k+1}p_x\cdot (k+1)\right)=l\cdot _lp_x-(n+l)\cdot _{n+l}p_x
\end{align*}
and
\begin{align*}
\int_{k}^{k+1}\frac{dt}{1-(k+1-t)\cdot q_{x+k}}=-\frac{\log p_{x+k}}{q_{x+k}}.
\end{align*}

Eq. \eqref{2} is a corollary of \eqref{T_x_m} too with $m=2$, where 
\begin{align*}
\sum_{k=l}^{l+n-1}\left(_kp_x\cdot k^2-_{k+1}p_x\cdot (k+1)^2\right)=l^2\cdot _lp_x-(n+l)^2\cdot _{n+l}p_x
\end{align*}
and
\begin{align*}
\int_{k}^{k+1}\frac{t\,dt}{1-(k+1-t)\cdot q_{x+k}}=\frac{q_{x+k}+(p_{x+k}-k\cdot q_{x+k})\cdot\log p_{x+k}}{q^2_{x+k}}.
\end{align*}
\end{proof}

\begin{proof}[Proof of Proposition \ref{increasing}]
According to the proof of Proposition \ref{prop:main}, we shall compute
\begin{align*}
\mathbb{E}\,T_x\cdot \nu^{T_x}\mathbbm{1}_{\{l\leqslant T_x \leqslant n+l\}}
&=
\sum_{k=l}^{l+n-1} {_{k+1}}p_x\cdot q_{x+k}\int\limits_{k}^{k+1}\frac{t\cdot\nu^t\, dt}{\left(1-(k+1-t)\cdot q_{x+k}\right)^2}.
\end{align*}

The last integral, multiplied by $q_{x+k}$, is

\begin{align*}
&I:=\frac{k\cdot \nu^k}{p_{x+k}}-(k+1)\cdot\nu^{k+1}+\int\limits_{k}^{k+1}\frac{\nu^t-\delta\cdot t \cdot \nu^t\,dt}{1-(k+1-t)\cdot q_{x+k}}
\\
&=\frac{k\cdot \nu^k}{p_{x+k}}-(k+1)\cdot\nu^{k+1}
+\frac{\nu^{1+k-1/q_{x+k}}}{q_{x+k}}
\left(
\mathrm{Ei}\left(-\frac{\delta}{q_{x+k}}\right)-\mathrm{Ei}\left(-\frac{\delta\cdot p_{x+k}}{q_{x+k}}\right)
\right)\\
&-\delta\int\limits_{k}^{k+1}\frac{t\cdot \nu^t\,dt}{1-(k+1-t)\cdot q_{x+k}}.
\end{align*}
Due to the change of variable $1-(k+1-t)q_{x+k}=y$ and rearrangements, the last integral, multiplied by $-\delta$, equals to
\begin{align*}
&-\frac{\delta\nu^{1+k-1/q_{x+k}}}{q_{x+k}}\int\limits_{p_{x+k}}^{1}\frac{(1+k-1/q_{x+k}+y/q_{x+k})\,\nu^{y/q_{x+k}}}{y/q_{x+k}}\,d\left(\frac{y}{q_{x+k}}\right)\\
&=-\delta\cdot\frac{\nu^{1+k-1/q_{x+k}}}{q_{x+k}}\left(\left(1+k-\frac{1}{q_{x+k}}\right)
\int\limits_{p_{x+k}/q_{x+k}}^{1/q_{x+k}}\frac{e^{-\delta z}}{z}\,dz+
\int\limits_{p_{x+k}/q_{x+k}}^{1/q_{x+k}}e^{-\delta z}\,dz
\right)\\
&=-\delta\cdot\frac{\nu^{1+k-1/q_{x+k}}}{q_{x+k}}
\left(
\left(1+k-\frac{1}{q_{x+k}}\right)
\left(\mathrm{Ei}\left(-\frac{\delta}{q_{x+k}}\right)-\mathrm{Ei}\left(-\frac{\delta\cdot p_{x+k}}{q_{x+k}}\right)\right)+\frac{i}{\delta}\nu^{1/q_{x+k}}
\right)\\
&=-\delta\cdot \frac{\nu^{1+k-1/q_{x+k}}}{q_{x+k}}\left(1+k-\frac{1}{q_{x+k}}\right)\left(\mathrm{Ei}\left(-\frac{\delta}{q_{x+k}}\right)-\mathrm{Ei}\left(-\frac{\delta\cdot p_{x+k}}{q_{x+k}}\right)\right)
-\frac{i\cdot \nu^{1+k}}{q_{x+k}}.
\end{align*}
Thus
\begin{align*}
I&=\frac{k\cdot \nu^k}{p_{x+k}}-\left(1+k+\frac{i}{q_{x+k}}\right)\nu^{k+1}\\
&+\left(1-\delta\left(1+k-\frac{1}{q_{x+k}}\right)\right)
\frac{\nu^{1+k-1/q_{x+k}}}{q_{x+k}}\left(\mathrm{Ei}\left(-\frac{\delta}{q_{x+k}}\right)-\mathrm{Ei}\left(-\frac{\delta\cdot p_{x+k}}{q_{x+k}}\right)\right),
\end{align*}
and finally 
\begin{align*}
&\mathbb{E}\,T_x\cdot \nu^{T_x}\mathbbm{1}_{\{l\leqslant T_x \leqslant n+l\}}
=\sum_{k=l}^{l+n-1}k\cdot _kp_x\cdot \nu^k-\sum_{k=l}^{l+n-1}\left(1+k+\frac{i}{q_{x+k}}\right)\cdot _{k+1}p_x\cdot \nu^{k+1}+\\
&\sum_{k=l}^{l+n-1}\left(1-\delta\left(1+k-\frac{1}{q_{x+k}}\right)\right)
\frac{_{k+1}p_x}{q_{x+k}}\nu^{1+k-\frac{1}{q_{x+k}}}\left(\mathrm{Ei}\left(-\frac{\delta}{q_{x+k}}\right)-\mathrm{Ei}\left(-\frac{\delta\cdot p_{x+k}}{q_{x+k}}\right)\right)\\
&=_lp_x\cdot l \cdot \nu^l-_{l+n}p_x\cdot (l+n) \cdot \nu^{l+n}-i\sum_{k=l}^{l+n-1}\frac{_{k+1}p_x}{q_{x+k}}\cdot \nu^{k+1}+\\
&\sum_{k=l}^{l+n-1}\left(1-\delta\left(1+k-\frac{1}{q_{x+k}}\right)\right)
\frac{_{k+1}p_x}{q_{x+k}}\nu^{1+k-\frac{1}{q_{x+k}}}\left(\mathrm{Ei}\left(-\frac{\delta}{q_{x+k}}\right)-\mathrm{Ei}\left(-\frac{\delta\cdot p_{x+k}}{q_{x+k}}\right)\right).
\end{align*}

We now compute the second moment
\begin{align*}
\mathbb{E}\,T_x^2\cdot \nu^{2T_x}\mathbbm{1}_{\{l\leqslant T_x \leqslant n+l\}}=
\sum_{k=l}^{l+n-1} {_{k+1}}p_x\cdot q_{x+k}\int\limits_{k}^{k+1}\frac{t^2\cdot\nu^{2t}\, dt}{\left(1-(k+1-t)\cdot q_{x+k}\right)^2}.
\end{align*}
The last integral, multiplied by $q_{x+k}$, is
\begin{align*}
\frac{k^2\cdot\nu^{2k}}{p_{x+k}}-(k+1)^2\cdot \nu^{2(k+1)}+
\int_{k}^{k+1}\frac{2t\cdot\nu^{2t}-2\delta\cdot t^2\cdot\nu^{2t}\,dt}{1-(k+1-t)\cdot q_{x+k}}.
\end{align*}
Let us denote 
\begin{align*}
I_{1,\,k}:=\int_{k}^{k+1}\frac{2t\cdot\nu^{2t}\,dt}{1-(k+1-t)\cdot q_{x+k}},
\quad
I_{2,\,k}=-\int_{k}^{k+1}\frac{2\delta\cdot t^2\cdot\nu^{2t}\,dt}{1-(k+1-t)\cdot q_{x+k}}.
\end{align*}
Then, omitting the details of elementary rearrangements, we obtain
\begin{align}\label{I_1}\nonumber
&I_{1,\,k}=\frac{\nu^{2k}}{q_{x+k}}\times\\
&\left(\frac{1-\nu^2}{\delta}+2\left(1+k-\frac{1}{q_{x+k}}\right)\nu^{2-\frac{2}{q_{x+k}}}\left(\mathrm{Ei}\left(-\frac{2\delta}{q_{x+k}}\right)-\mathrm{Ei}\left(-\frac{2\delta\cdot p_{x+k}}{q_{x+k}}\right)\right)\right),
\end{align}
\begin{align}\label{I_2}
&I_{2,\,k}=\frac{\nu^{2k}}{q_{x+k}}\Bigg(\frac{1-\nu^2}{q_{x+k}}
+\frac{-1+\nu^2+\delta\cdot (-2+4\nu^2+4k(-1+\nu^2))}{2\delta}
\nonumber\\
&-2\delta\left(1+k-\frac{1}{q_{x+k}}\right)^2\nu^{2-2/q_{x+k}}
\left(
\mathrm{Ei}\left(-\frac{2\delta}{q_{x+k}}\right)-\mathrm{Ei}\left(-\frac{2\delta\cdot p_{x+k}}{q_{x+k}}\right)
\right)
\Bigg).
\end{align}
Thus,
\begin{align*}
&\mathbb{E}\,T_x^2\cdot \nu^{2T_x}\mathbbm{1}_{\{l\leqslant T_x \leqslant n+l\}}\\
&=l^2\cdot \nu^{2l}\cdot _lp_x-(l+n)^2\cdot\nu^{2(l+n)}\cdot _{l+n}p_x+
\sum_{k=l}^{n+l-1}\, _{k+1}p_x\cdot \left(I_{1,\,k}+I_{2,\,k}\right)
\end{align*}
\end{proof}

\begin{proof}[Proof of Proposition \ref{prop:[T]_nu}]
Arguing the same as before, we get
\begin{align*}
&\mathbb{E}\,([T_x+1]\nu^{T_x})^m\mathbbm{1}_{\{l\leqslant T_x\leqslant n+l\}}\\
&=\sum_{k=l}^{l+n-1} {_{k+1}}p_x \cdot q_{x+k} \cdot (k+1)^m \cdot \int_{k}^{k+1}\frac{\nu^{tm}\,dt}{(1-(k+1-t)\cdot q_{x+k})^2}\\
&=\sum_{k=l}^{l+n-1}{_{k+1}}p_x\cdot (k+1)^m
\left(
\frac{\nu^{mk}}{p_{x+k}}-\nu^{m(k+1)}+\delta m \cdot\frac{\nu^{m(1+k-1/q_{x+k})}}{q_{x+k}}\cdot \mathrm{Ei}(m\delta)
\right)\\
&=\sum_{k=l}^{l+n-1} {_k}p_x\cdot (k+1)^m \cdot \nu^{mk}\cdot(1-p_{x+k}\cdot \nu^m)\\
&+\delta m \sum_{k=l}^{l+n-1} {_{k+1}}p_x\cdot (k+1)^m\cdot \frac{v^{m(1+k-1/q_{x+k})}}{q_{x+k}}\cdot \mathrm{Ei}_k(m\delta).
\end{align*}
\end{proof}

\begin{proof}[Proof of Proposition \ref{prop:j_times}]

The proof is straightforward because of Lemma \ref{lem:j_times}

\begin{align*}
&\mathbb{E}\left(\nu^\frac{[T_xj]+1}{j}\right)^m\mathbbm{1}_{\{l*n_1\leqslant T_x<n+l*n_1\}}
=\sum_{k=l}^{n+l-1}\sum_{d=n_1}^{j-1}\nu^{(k+(d+1)/j)m}\left(_{k+\frac{d}{j}}p_x-_{k+\frac{d+1}{j}}p_x\right)\\
&+
\nu^{(n+l)m}\sum_{d=0}^{n_1-1}\nu^{(d+1)/j\cdot m}\left(_{n+l+\frac{d}{j}}p_x-_{n+l+\frac{d+1}{j}}p_x\right)\\
&=\frac{1}{j}\sum_{k=l}^{n+l-1}\sum_{d=n_1}^{j-1}\nu^{(k+(d+1)/j)m}\frac{_{k+1}p_x\cdot q_{x+k}}{\left(p_{x+k}+\frac{d}{j}\cdot q_{x+k}\right)\left(p_{x+k}+\frac{d+1}{j}\cdot q_{x+k}\right)}\\
&+\frac{\nu^{(n+l)m}}{j}\cdot{_{n+l+1}}p_x\cdot q_{x+n+l} \sum_{d=0}^{n_1-1}
\frac{\nu^{(d+1)/j\cdot m}}{\left(p_{x+n+l}+\frac{d}{j}\cdot q_{x+n+l}\right)
\left(p_{x+n+l}+\frac{d+1}{j}\cdot q_{x+n+l}\right)
}.
\end{align*}
\end{proof}

\begin{proof}[Proof of Proposition \ref{prop:j_[T]_nu_x}]
If $j\in\mathbb{N}$ and $d\in\{0,\,1,\,\ldots,\,j-1\}$ then
\begin{align}\nonumber
&\mathbb{E}\left([j\cdot T_x+1]\cdot \nu^{T_x}\right)^m\mathbbm{1}_{\{ l*n_1\leqslant T_x< n+l*n_1\}}\\
&=\sum_{k=l}^{n+l-1}\sum_{d=n_1}^{j-1}(d+j\cdot k+1)^m\int\limits_{k+d/j}^{k+(d+1)/j}\nu^{m t} f_x(t)\,dt\label{int_1}\\
&+\sum_{d=0}^{n_1-1}(d+j\cdot(n+l)+1)^m\int\limits_{n+l+d/j}^{n+l+(d+1)/j}\nu^{mt}f_x(t)\,dt.
\label{int_2}
\end{align}
The integral from \eqref{int_1} is
\begin{align*}
&\frac{1}{_{k+1}p_x}\int\limits_{k+d/j}^{k+(d+1)/j}\nu^{m t} f_x(t)\,dt
=q_{x+k}\int\limits_{k+\frac{d}{j}}^{k+\frac{d+1}{j}}\frac{\nu^{mt}\,dt}{(1-(k+1-t)q_{x+k})^2}\\
&=\nu^{m\left(k+\frac{d}{j}\right)}\left(\frac{1}{1-(1-d/j)\cdot q_{x+k}}-\frac{\nu^{m/j}}{1-(1-(d+1)/j)\cdot q_{x+k}}\right)\\
&-\delta m\int\limits_{k+d/j}^{k+(d+1)/j}\frac{e^{-\delta m t}\,dt}{1-(k+1-t)q_{x+k}},
\end{align*}
while the integral from \eqref{int_2} is the same if $k=n+l$. The proof follows upon Lemma \ref{lem:spec_int} and other elementary means.
\end{proof}

\section{Acknowledgments}
We thank our colleague Professor Jonas Šiaulys for reviewing the first draft of the manuscript. 

\bibliography{sn-bibliography}


\begin{thebibliography}{23}
\ifx \bisbn   \undefined \def \bisbn  #1{ISBN #1}\fi
\ifx \binits  \undefined \def \binits#1{#1}\fi
\ifx \bauthor  \undefined \def \bauthor#1{#1}\fi
\ifx \batitle  \undefined \def \batitle#1{#1}\fi
\ifx \bjtitle  \undefined \def \bjtitle#1{#1}\fi
\ifx \bvolume  \undefined \def \bvolume#1{\textbf{#1}}\fi
\ifx \byear  \undefined \def \byear#1{#1}\fi
\ifx \bissue  \undefined \def \bissue#1{#1}\fi
\ifx \bfpage  \undefined \def \bfpage#1{#1}\fi
\ifx \blpage  \undefined \def \blpage #1{#1}\fi
\ifx \burl  \undefined \def \burl#1{\textsf{#1}}\fi
\ifx \doiurl  \undefined \def \doiurl#1{\url{https://doi.org/#1}}\fi
\ifx \betal  \undefined \def \betal{\textit{et al.}}\fi
\ifx \binstitute  \undefined \def \binstitute#1{#1}\fi
\ifx \binstitutionaled  \undefined \def \binstitutionaled#1{#1}\fi
\ifx \bctitle  \undefined \def \bctitle#1{#1}\fi
\ifx \beditor  \undefined \def \beditor#1{#1}\fi
\ifx \bpublisher  \undefined \def \bpublisher#1{#1}\fi
\ifx \bbtitle  \undefined \def \bbtitle#1{#1}\fi
\ifx \bedition  \undefined \def \bedition#1{#1}\fi
\ifx \bseriesno  \undefined \def \bseriesno#1{#1}\fi
\ifx \blocation  \undefined \def \blocation#1{#1}\fi
\ifx \bsertitle  \undefined \def \bsertitle#1{#1}\fi
\ifx \bsnm \undefined \def \bsnm#1{#1}\fi
\ifx \bsuffix \undefined \def \bsuffix#1{#1}\fi
\ifx \bparticle \undefined \def \bparticle#1{#1}\fi
\ifx \barticle \undefined \def \barticle#1{#1}\fi
\bibcommenthead
\ifx \bconfdate \undefined \def \bconfdate #1{#1}\fi
\ifx \botherref \undefined \def \botherref #1{#1}\fi
\ifx \url \undefined \def \url#1{\textsf{#1}}\fi
\ifx \bchapter \undefined \def \bchapter#1{#1}\fi
\ifx \bbook \undefined \def \bbook#1{#1}\fi
\ifx \bcomment \undefined \def \bcomment#1{#1}\fi
\ifx \oauthor \undefined \def \oauthor#1{#1}\fi
\ifx \citeauthoryear \undefined \def \citeauthoryear#1{#1}\fi
\ifx \endbibitem  \undefined \def \endbibitem {}\fi
\ifx \bconflocation  \undefined \def \bconflocation#1{#1}\fi
\ifx \arxivurl  \undefined \def \arxivurl#1{\textsf{#1}}\fi
\csname PreBibitemsHook\endcsname

\bibitem[\protect\citeauthoryear{Batten}{1978}]{batten1978mortality}
\begin{bbook}
\bauthor{\bsnm{Batten}, \binits{R.W.}}:
\bbtitle{Mortality Table Construction}.
\bsertitle{Risk, Insurance and Security Series}.
\bpublisher{Prentice-Hall},
\blocation{New Jersey}
(\byear{1978})
\end{bbook}
\endbibitem

\bibitem[\protect\citeauthoryear{Bowers(Jr.)
  et~al.}{1986}]{bowers1986actuarial}
\begin{bbook}
\bauthor{\bsnm{Bowers(Jr.)}, \binits{N.L.}},
\bauthor{\bsnm{Gerber}, \binits{H.U.}},
\bauthor{\bsnm{Hickman}, \binits{J.C.}},
\bauthor{\bsnm{Jones}, \binits{D.A.}},
\bauthor{\bsnm{Nesbitt}, \binits{C.J.}}:
\bbtitle{Actuarial Mathematics},
\bedition{1}st edn.
\bpublisher{Society of Actuaries},
\blocation{Itasca}
(\byear{1986})
\end{bbook}
\endbibitem

\bibitem[\protect\citeauthoryear{Einaudi}{1993}]{Balducci_years}
\begin{bbook}
\bauthor{\bsnm{Einaudi}, \binits{L.}}:
\bbtitle{Diario 1945-1947}.
\bpublisher{Editori Laterza},
\blocation{Italy}
(\byear{1993}).
\burl{https://www.bancaditalia.it/pubblicazioni/collana-storica/documenti/documenti-11/CSBI-documenti-11.pdf}
\end{bbook}
\endbibitem

\bibitem[\protect\citeauthoryear{Mosca}{}]{mosca_artbilancio}
\begin{botherref}
\oauthor{\bsnm{Mosca}, \binits{M.}}:
Precedenti Storici e Conseguenze Economiche Dell’articolo 81 della
  Costituzione.
\url{http://manuelamosca.com/progetti/artbilancio.pdf}
\end{botherref}
\endbibitem

\bibitem[\protect\citeauthoryear{Unknown}{1952}]{balducci1952}
\begin{botherref}
\oauthor{\bsnm{Unknown}}:
Gaetano Balducci, ragionere generale dello Stato, 1952.
Italian magazine Epoca, Vol. VII, n. 86, 31 May 1952
(1952)
\end{botherref}
\endbibitem

\bibitem[\protect\citeauthoryear{Gaetano}{1911}]{balducci1911tavola}
\begin{barticle}
\bauthor{\bsnm{Gaetano}, \binits{B.}}:
\batitle{La tavola di sopravvivenza della popolazione maschile italiana (1901)
  interpolata mediante la formola del makeham}.
\bjtitle{Giornale degli Economisti e Rivista di Statistica}
\bvolume{42}(\bissue{4}),
\bfpage{394}--\blpage{403}
(\byear{1911})
\end{barticle}
\endbibitem

\bibitem[\protect\citeauthoryear{Balducci}{1917}]{balducci1917}
\begin{barticle}
\bauthor{\bsnm{Balducci}, \binits{G.}}:
\batitle{Costruzione e critica delle tavole di mortalità}.
\bjtitle{Giornale degli Economisti e Rivista di Statistica}
\bvolume{55 (Anno 28)}(\bissue{6}),
\bfpage{455}--\blpage{484}
(\byear{1917})
\end{barticle}
\endbibitem

\bibitem[\protect\citeauthoryear{Beyer}{1903}]{Theodor}
\begin{bbook}
\bauthor{\bsnm{Beyer}, \binits{O.W.}}:
\bbtitle{Deutsche Schulwelt des Neunzehnten Jahrhunderts in Wort und Bild}.
\bpublisher{Pichler},
\blocation{Leipzig und Wien}
(\byear{1903}).
\burl{https://scripta.bbf.dipf.de/viewer/!image/481079971/366/-/}
\end{bbook}
\endbibitem

\bibitem[\protect\citeauthoryear{Wittstein}{1862}]{wittstein1862}
\begin{barticle}
\bauthor{\bsnm{Wittstein}, \binits{T.}}:
\batitle{Die mortalit{\"a}t in gesellschaften mit successiv eintretenden und
  ausscheidenden mitgliedern}.
\bjtitle{Archiv der Mathematik und Physik}
\bvolume{39},
\bfpage{67}--\blpage{92}
(\byear{1862})
\end{barticle}
\endbibitem

\bibitem[\protect\citeauthoryear{Puzey}{1993}]{puzey1993}
\begin{botherref}
\oauthor{\bsnm{Puzey}, \binits{G.}}:
New techniques for the development of general-purpose parallel programming
  systems.
Phd thesis,
City University London
(1993).
\url{https://openaccess.city.ac.uk/id/eprint/29597/1/Puzey%20thesis%201993%20PDF-A.pdf}
\end{botherref}
\endbibitem

\bibitem[\protect\citeauthoryear{Inc.}{}]{Mathematica}
\begin{botherref}
\oauthor{\bsnm{Inc.}, \binits{W.R.}}:
Mathematica, {V}ersion 14.0.
Champaign, IL, 2024.
\url{https://www.wolfram.com/mathematica}
\end{botherref}
\endbibitem

\bibitem[\protect\citeauthoryear{Wolthuis}{2014}]{notations}
\begin{bbook}
\bauthor{\bsnm{Wolthuis}, \binits{H.}}:
\bbtitle{International Actuarial Notation}.
\bpublisher{John Wiley \& Sons, Ltd}, \blocation{???}
(\byear{2014}).
\doiurl{10.1002/9781118445112.stat04347}
\end{bbook}
\endbibitem

\bibitem[\protect\citeauthoryear{Giao}{2003}]{exp_1}
\begin{barticle}
\bauthor{\bsnm{Giao}, \binits{P.H.}}:
\batitle{Revisit of well function approximation and an easy graphical curve
  matching technique for theis' solution}.
\bjtitle{Groundwater}
\bvolume{41}(\bissue{3}),
\bfpage{387}--\blpage{390}
(\byear{2003})
\doiurl{10.1111/j.1745-6584.2003.tb02608.x}
\end{barticle}
\endbibitem

\bibitem[\protect\citeauthoryear{Tseng and Lee}{1998}]{exp_2}
\begin{barticle}
\bauthor{\bsnm{Tseng}, \binits{P.-H.}},
\bauthor{\bsnm{Lee}, \binits{T.-C.}}:
\batitle{Numerical evaluation of exponential integral: Theis well function
  approximation}.
\bjtitle{Journal of Hydrology}
\bvolume{205}(\bissue{1}),
\bfpage{38}--\blpage{51}
(\byear{1998})
\doiurl{10.1016/S0022-1694(97)00134-0}
\end{barticle}
\endbibitem

\bibitem[\protect\citeauthoryear{Barry et~al.}{2000}]{exp_3}
\begin{barticle}
\bauthor{\bsnm{Barry}, \binits{D.A.}},
\bauthor{\bsnm{Parlange}, \binits{J.-Y.}},
\bauthor{\bsnm{Li}, \binits{L.}}:
\batitle{Approximation for the exponential integral (theis well function)}.
\bjtitle{Journal of Hydrology}
\bvolume{227}(\bissue{1}),
\bfpage{287}--\blpage{291}
(\byear{2000})
\doiurl{10.1016/S0022-1694(99)00184-5}
\end{barticle}
\endbibitem

\bibitem[\protect\citeauthoryear{Andrews et~al.}{1999}]{Specia_F}
\begin{bbook}
\bauthor{\bsnm{Andrews}, \binits{G.E.}},
\bauthor{\bsnm{Askey}, \binits{R.}},
\bauthor{\bsnm{Roy}, \binits{R.}}:
\bbtitle{The Hypergeometric Functions}.
\bsertitle{Encyclopedia of Mathematics and its Applications},
pp. \bfpage{61}--\blpage{123}.
\bpublisher{Cambridge University Press},
\blocation{Cambridge}
(\byear{1999})
\end{bbook}
\endbibitem

\bibitem[\protect\citeauthoryear{Bravo}{2007}]{bravo2007tabuas}
\begin{botherref}
\oauthor{\bsnm{Bravo}, \binits{J.}}:
T{\'a}buas de mortalidade contempor{\^a}neas e prospectivas: Modelos
  estoc{\'a}sticos, aplica{\c{c}}{\~o}es actuariais e cobertura do risco de
  longevidade.
Universidade de {\'E}vora, {\'E}vora
(2007)
\end{botherref}
\endbibitem

\bibitem[\protect\citeauthoryear{Castellares et~al.}{2022}]{Brazilian}
\begin{barticle}
\bauthor{\bsnm{Castellares}, \binits{F.}},
\bauthor{\bsnm{Patr{\'i}cio}, \binits{S.}},
\bauthor{\bsnm{Lemonte}, \binits{A.J.}}:
\batitle{{On the Gompertz–Makeham law: A useful mortality model to deal with
  human mortality}}.
\bjtitle{Brazilian Journal of Probability and Statistics}
\bvolume{36}(\bissue{3}),
\bfpage{613}--\blpage{639}
(\byear{2022})
\doiurl{10.1214/22-BJPS545}
\end{barticle}
\endbibitem

\bibitem[\protect\citeauthoryear{Richards}{2012}]{AMJ}
\begin{barticle}
\bauthor{\bsnm{Richards}, \binits{S.J.}}:
\batitle{A handbook of parametric survival models for actuarial use}.
\bjtitle{Scandinavian Actuarial Journal}
\bvolume{2012}(\bissue{4}),
\bfpage{233}--\blpage{257}
(\byear{2012})
\doiurl{10.1080/03461238.2010.506688}
\end{barticle}
\endbibitem

\bibitem[\protect\citeauthoryear{Amarjit~Kundu and Balakrishnan}{2023}]{CMSTM}
\begin{barticle}
\bauthor{\bsnm{Amarjit~Kundu}, \binits{S.C.}},
\bauthor{\bsnm{Balakrishnan}, \binits{N.}}:
\batitle{Ordering properties of the smallest and largest lifetimes in
  gompertz–makeham model}.
\bjtitle{Communications in Statistics - Theory and Methods}
\bvolume{52}(\bissue{3}),
\bfpage{643}--\blpage{669}
(\byear{2023})
\doiurl{10.1080/03610926.2021.1919898}
\end{barticle}
\endbibitem

\bibitem[\protect\citeauthoryear{Puišys et~al.}{2024}]{RSJ}
\begin{barticle}
\bauthor{\bsnm{Puišys}, \binits{R.}},
\bauthor{\bsnm{Lewkiewicz}, \binits{S.}},
\bauthor{\bsnm{Šiaulys}, \binits{J.}}:
\batitle{Properties of the random effect transformation}.
\bjtitle{Lithuanian Mathematical Journal}
\bvolume{64},
\bfpage{177}--\blpage{189}
(\byear{2024})
\doiurl{10.1007/s10986-024-09633-3}
\end{barticle}
\endbibitem

\bibitem[\protect\citeauthoryear{Kreer et~al.}{2024}]{Weibull}
\begin{barticle}
\bauthor{\bsnm{Kreer}, \binits{M.}},
\bauthor{\bsnm{Kizilersu}, \binits{A.}},
\bauthor{\bsnm{Thomas}, \binits{A.W.}}:
\batitle{When is the discrete weibull distribution infinitely divisible?}
\bjtitle{Statistics \& Probability Letters}
\bvolume{215},
\bfpage{110}--\blpage{238}
(\byear{2024})
\doiurl{10.1016/j.spl.2024.110238}
\end{barticle}
\endbibitem

\bibitem[\protect\citeauthoryear{Endo et~al.}{2022}]{Weibull_1}
\begin{barticle}
\bauthor{\bsnm{Endo}, \binits{A.}},
\bauthor{\bsnm{Murayama}, \binits{H.}},
\bauthor{\bsnm{Abbott}, \binits{S.}},
\bauthor{\bsnm{Ratnayake}, \binits{R.}},
\bauthor{\bsnm{Pearson}, \binits{C.A.B.}},
\bauthor{\bsnm{Edmunds}, \binits{W.J.}},
\bauthor{\bsnm{Fearon}, \binits{E.}},
\bauthor{\bsnm{Funk}, \binits{S.}}:
\batitle{Heavy-tailed sexual contact networks and monkeypox epidemiology in the
  global outbreak, 2022}.
\bjtitle{Science}
\bvolume{378}(\bissue{6615}),
\bfpage{90}--\blpage{94}
(\byear{2022})
\doiurl{10.1126/science.add4507}
\end{barticle}
\endbibitem

\end{thebibliography}

\end{document}